\documentclass[journal,10pt,onecolumn,draftclsnofoot,]{IEEEtran}

\IEEEoverridecommandlockouts                              

\usepackage{amsmath, amsfonts, amssymb, mathrsfs, setspace, graphicx, epsfig, amsthm, caption,subcaption,url,float,framed,circuitikz, color,dsfont}
\usepackage{hyperref}
\newtheorem{theorem}{Theorem}

\newtheorem{definition}{Definition}
\newtheorem{lemma}{Lemma}
\newtheorem{corollary}{Corollary}

\newtheorem{remark}{Remark}

\makeatletter
\usepackage[noadjust]{cite}
\makeatother

\title{\huge Epidemic Processes over Time--Varying Networks }

\author{Philip E. Par\'{e},
Carolyn L. Beck, and
Angelia Nedi\'{c}*\thanks{
* Philip E. Par\'{e}, Carolyn L. Beck, and Angelia Nedi\'{c} are with the Coordinated Science Laboratory at the University of Illinois at Urbana-Champaign and can be reached at {\tt philip.e.pare@gmail.com} (corresponding author), {\tt beck3@illinois.edu}, and {\tt angelia@illinois.edu}, respectively and 1308 W Main St, Urbana, IL 61801.  This material is based on research partially sponsored by the National Science Foundation, grants ECCS 15-09302, CCF 0904619, and DMS 13-12907.  All material in this paper represents the position of the authors and not necessarily that of the NSF.
}}
\begin{document}
\maketitle

\begin{abstract}
The spread of viruses in  biological networks,  computer networks, and human contact networks can have devastating effects; developing and analyzing mathematical models of these systems can be insightful and lead to societal benefits. Prior research has focused mainly on  network models with static graph structures, however the systems being modeled typically have dynamic graph structures. Therefore to better understand and analyze virus spread, further study is required.  In this paper, we consider  virus spread models over networks with dynamic graph structures, and investigate the behavior of diseases in these systems. A stability analysis of  epidemic processes over time--varying networks is performed, examining conditions for the disease free equilibrium, in both the deterministic and stochastic cases. We present  simulation results, propose a number of corollaries based on these simulations, and discuss quarantine control via simulation.
\end{abstract}

\section{Introduction}


Mathematical models of epidemics have been studied for 
hundreds of years. Bernoulli developed one of the first known models inspired by the smallpox virus \cite{bernoulli1760essai}, while at the same time other models were being developed \cite{dietz2002bernoulli}.
Of particular interest to the work discussed herein are the so-called 
susceptible-infected-susceptible (SIS) models, which have been developed for both continuous  \cite{kermack1932contributions,fall2007epidemiological,van2009virus,ahn2013global,pare2015stability} and discrete  time domains \cite{ahn2013global,wang2003epidemic,peng2010epidemic}.  These are standard models commonly used to capture
the evolution of virus infections in networks.  In a basic SIS model, at each time step each individual node or agent is either {\em infected}, or {\em susceptible} to infection.  A susceptible agent may be infected by neighboring agents with some given infection rate, $\beta$, where the network graph structure determines the
connectivity between agents, and hence plays a direct role in facilitating or inhibiting the spread of infection.  An infected agent may be cured, returning to the susceptible state, with some given healing rate, $\delta$. The first SIS model, developed by Kermack and McKendrick \cite{kermack1932contributions}, is given by
\begin{equation}\label{eq:sis}
\begin{split}
\dot{S}(t) & =  -\beta S(t)I(t) + \delta I(t) \\
\dot{I}(t) & =  \beta S(t)I(t) - \delta I(t) ,
\end{split}
\end{equation}
where $S(t)$ is the group of susceptible agents and $I(t)$ is the group of infected agents. This model considers the propagation of a virus over a trivial network, that is, it assumes complete connectivity, and models the infected and susceptible agents as two aggregated groups.

In Wang, {\em et al.} \cite{wang2003epidemic}, a discrete time model for virus spread over a nontrivial network is proposed, and an epidemic threshold is derived, guaranteeing convergence to the disease free equilibrium (DFE). The convergence rate is given in terms of the largest eigenvalue of the adjacency matrix of the network, which is proportional to the ratio of the healing and infection rates.  A necessary and sufficient condition for exponential stability of the DFE is provided, with several simulations and results for specific graph structures given. In Peng, {\em et al.} \cite{peng2010epidemic}, the authors provide a necessary and sufficient condition for stability of  the DFE for an extension of the model from \cite{wang2003epidemic}.  Peng, {\em et al.} \cite{peng2010epidemic} also discuss possible immunization techniques, and formulate a convex optimization problem for forming an optimal immunization strategy, based on a relaxation of the problem constraints. In Ahn and Hassibi \cite{ahn2013global}, the authors study both discrete and continuous time SIS models.  Both the DFE and the non-disease free equilibrium (NDFE) of several models are considered, and existence, uniqueness, and stability conditions of the NDFE are established. In Fall, {\em et al.} \cite{fall2007epidemiological}, a continuous time SIS model is analyzed, with sufficient conditions for global asymptotic stability derived for both the DFE and the NDFE. In Van Mieghem, {\em et al.} \cite{van2009virus}, the authors expose some limitations of the model presented in \cite{wang2003epidemic}, namely that it is only accurate if the spread rate is below the epidemic threshold. They propose a $2^n$-state Markov chain model and an $n$-intertwined Markov chain model as alternatives to the model in \cite{wang2003epidemic} and explore the properties of these models by studying the limiting cases of the complete graph and the line graph. These two models will be explained in more detail in the following section. Geometric programming ideas \cite{boyd2007tutorial} are used in Preciado, {\em et al.} \cite{preciado2013optimal,preciado2014optimal}, to develop optimal vaccination techniques for the continuous time model in \cite{van2009virus}. 
In Pasqualetti, {\em et al.} \cite{bullo2014control}, a network control technique is applied to a discretized, linearized version of the model from \cite{van2009virus}. In Khanafer, {\em et al.} \cite{khanafer2014stability,Khanafer15arxiv}, the stability properties of the equilibria of the continuous time model from \cite{van2009virus} are further explored with an antidote control technique proposed. For a more complete survey of this area see \cite{nowzari2016analysis}.

The aforementioned papers consider only static graph structures.  Although contributing to the understanding of virus spread and the ensuing eradication,
in most cases static graph structures are fundamentally too simple to capture the essential dynamics of infectious disease processes. Most applications that motivate these systems have agents that are mobile, which implies that the underlying graph structure is time--varying. For example, computer networks are comprised of smart phones, laptops, and other mobile devices which connect to different devices as they move around. 
In order to obtain a better understanding of these systems, a more realistic representation, 
that of {\em virus dynamics over time--varying networks}, is necessary.

The previous work on virus spread over time--varying networks is limited to unweighted and undirected graphs. In Prakash, {\em et al.} \cite{prakash2010virus}, the authors extend the discrete time model used in \cite{wang2003epidemic} to a model with a time--varying graph structure. They provide a sufficient condition for local exponential stability of the origin (the DFE), and propose a control scheme that removes immune agents from the system. 
In Bokharaie, {\em et al.} \cite{bokharaie2010spread}, the authors provide similar results to those in \cite{prakash2010virus}, extending the model from \cite{wang2003epidemic} to the time--varying case and proving local exponential convergence to the origin.  Bokharaie, {\em et al.}  further state that these results extend to the continuous time--varying case, without proof. Both \cite{prakash2010virus} and \cite{bokharaie2010spread} consider unweighted, undirected graphs and extend the model from \cite{wang2003epidemic}, which was shown in \cite{van2009virus} to have significant shortcomings. 
In Rami, {\em et al.}  \cite{rami2014switch}, the authors extend the model from \cite{fall2007epidemiological} to that of a switching virus model. A sufficient condition for stability of the DFE is provided, existence of a periodic NDFE is shown, and  a sufficient condition for stability of the DFE for a Markovian switching virus model is given.

In this paper, we propose notable extensions to the models from \cite{van2009virus} to include   time--varying, weighted, undirected and directed graph structures and provide sufficient conditions for global exponential stability of the DFE for these models. 
We compare the $2^n$-state Markov chain model and $n$-intertwined Markov chain model via simulations for both static graphs and time--varying graphs.
We present various simulations that give insight into stability and  quarantine control results, and use these to motivate  several corollaries and remarks. 
A preliminary version of these results is given in \cite{pare2015stability}. Additional contributions of this paper include: 1) the development of conditions for global exponential convergence to the DFE for heterogeneous viruses over directed time--varying networks, 
2)  extensions of the time--varying virus models to a {\em stochastic framework} with accompanying stability results, and 3) an in-depth evaluation of the effectiveness of the $n$-intertwined Markov chain model as a mean field approximation of the $2^n$ model via simulation.

The paper is organized as follows. We first introduce our notation. In Section \ref{sec:model} we introduce the models from \cite{van2009virus} with  the mean field approximation derivation of the $n$-intertwined Markov chain model from  the $2^n$-state Markov chain model, and present the time--varying extensions of these models. In Section \ref{sec:stab} we explore the stability properties of the time--varying extensions of the $n$-intertwined Markov chain model. In Section \ref{sec:sim} we evaluate the effectiveness of the mean field approximations of  both the static and time--varying $n$-intertwined models by comparing them to the $2^n$ model via simulation. Based on these simulations and others, we state and prove several corollaries. 
Finally, in Section \ref{sec:con} we conclude, summarizing the results and discussing future work.

\subsection{Notation}

Given a vector function of time $x(t)$, $\dot{x}(t)$ indicates the time-derivative. Given a vector $x \in \mathbb{R}^{n}$, the 2-norm is denoted by $\|x\|$ and the transpose by $x^T$. Given a matrix $A \in \mathbb{R}^{n \times n}$, the maximum eigenvalue 
is $\lambda_1(A)$ (if the spectrum is all real), the largest real-valued part of the eigenvalues of $A$ 
is $s_1(A)$ (if the spectrum is possibly complex), 
$a_{ij}$ indicates the $i, j^{th}$ entry of $A$, and $\| A \|$ indicates the induced 2-norm of $A$. The notation $diag(\cdot)$ refers to a diagonal matrix with the argument on the diagonal. 
We use $E[\cdot]$ to denote the expected value of the argument and $Pr[\cdot]$ to denote the probability of the argument. 

\section{The Virus Model} \label{sec:model}




In \cite{van2009virus},  a $2^n$-state Markov chain is introduced, where  each state of the chain, $Y_k(t)$, corresponds to a binary vector  $x \in \mathbb{R}^n$ 
and the state transition matrix, $Q$, is defined by
\begin{equation}\label{eq:qij}
q_{kl}=
\begin{cases}
\delta, & \text{ if } x_i = 1, k = l + 2^{i-1}\\
\beta \displaystyle\sum_{j=1}^n a_{ij}  x_j, &\text{ if } x_i = 0, k = l - 2^{i-1} \\[2.5ex]
-\displaystyle\sum_{j\neq l} q_{jl}, & \text{ if } k = l\\
0, & \text{ otherwise,}
\end{cases}
\end{equation}
for $i = 1,...,n$. Here a virus is propagating over a network defined by $a_{ij}$ (non-negative and $a_{jj}=0 \ \forall j$), with $n$ agents, or possibly $n$ groupings of agents, $\beta$ is the infection rate, $\delta$ is the healing rate, and  $x_i=1$ or $x_i=0$ indicates that the $i$th agent is either infected or susceptible, respectively. 

\begin{figure}
    \centering
    \includegraphics[width=.5\columnwidth]{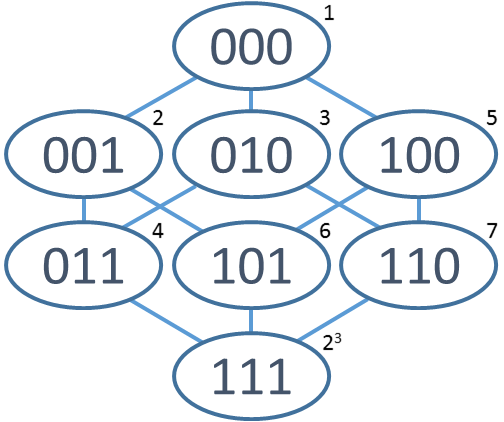}
    \caption{Example of $2^n$-state model with $n=3$: the superscripts indicate the ordering of the states, which correspond to the subscript of $y_k(t)$ in \eqref{eq:y}, and the internal strings indicate which agents are healthy (0) and which are infected (1), corresponding to $x_i$ in \eqref{eq:qij}.}
    \label{fig:2to3}
\end{figure}
The state vector $y(t)$ is defined as 
\begin{equation}\label{eq:y}
         y_k(t) = Pr[Y_k(t) = k], 
\end{equation}
with $\sum_{k=1}^{2^n} y_k(t) =1$. The Markov chain evolves as
\begin{equation}\label{eq:2n}
    \frac{dy^T(t)}{dt} = y^T(t)Q.
\end{equation}
See Figure \ref{fig:2to3} for an example of this chain with $n=3$.
Let  $v_i(t) = Pr[X_i(t) = 1]$, where $X_i(t)$ is the random variable representing whether the $i$th agent is infected, then
\begin{equation}\label{eq:v}
    v^T(t) = y^T(t)M,
\end{equation}
where $M \in \mathbb{R}^{2^n\times n}$ with the rows being lexicographically-ordered binary numbers, bit reversed\footnote{Matlab code: $M = fliplr(dec2bin(0:(2^n)-1)-'0')$}. That is, $v_i(t)$ reflects the summation of all probabilities where $x_i = 1$, therefore giving the mean, $E[X_i]$, of the infection, $X_i$, of agent $i$. 

A mean field approximation of this system is used to obtain an $n$-intertwined continuous Markov chain, where each node has two states: infected with probability $Pr[X_i(t)=1]$ and healthy (susceptible) with probability $Pr[X_i(t)=0]$. Taking the expected value of the second case of \eqref{eq:qij}, that is, the infection transition rate, 
and using $E[1_{z}] = Pr[z]$ gives
\begin{align}\label{eq:mean}
        E[q_{kl}| x_i = 0, k = l - 2^{i-1}] &= E\big[\beta \displaystyle\sum_{l=1}^n a_{ij}  1_{\{X_j(t)=1\}}\big]\\ 
        &= \beta \displaystyle\sum_{l=1}^n a_{ij}Pr[X_j(t)=1],\nonumber
\end{align}
where the second equality holds since the $\beta$ and $a_{ij}$ values are deterministic and known.

Denoting $p_i(t) = Pr[X_i(t)=1]$ and noting that $Pr[X_i(t)=0] = 1 - p_i(t)$, we can see that
\begin{equation}\label{eq:p}
\dot{p}_i(t) = (1 - p_i(t))\beta \sum_{j=1}^n a_{ij}  p_j(t) - \delta p_i(t).
\end{equation}
Applying the Central Limit Theorem, under the assumption of independent indicators, implies that large deviations from the mean are unlikely; this is the motivation for the mean field approximation. However, it is clear that the indicators are not  independent, by construction. The authors of \cite{van2009virus} state 
\begin{equation*}
Pr[X_j(t)=1 | X_i(t) = 1 ] \geq Pr[X_j(t) = 1],
\end{equation*}
which is true under the assumption $\beta \geq 0$, because if one node in the system is infected, it will have only a non-negative effect on the probability of infecting other nodes. That is, one agent being infected will never decrease the probability of another agent becoming infected. 
Therefore, the $n$-intertwined Markov chain model gives an {\em upper-bound} for the exact probability of infection $p_i(t)$ \cite{van2009virus}. It has been shown that, under certain conditions, mean field approximations of SIS models may be inaccurate, leading to incorrect results \cite{chatterjee2009contact}. However, as we have shown, the mean field approximation considered herein, while it is an approximation,  is well constructed. The shortcomings of this mean field approximation are illustrated in Section \ref{sec:sim}.

The model in \eqref{eq:p} can be generalized to the heterogeneous virus, directed graph structure case:
\begin{equation}\label{eq:vani}
\dot{p}_i(t) = (1 - p_i(t))\beta_i \sum_{j=1}^n a_{ij} p_j(t) - \delta_i p_i(t),
\end{equation}
where 
$\beta_i$ is the non-negative susceptibility or infection rate of agent $i$; $n$ is the number of agents; $a_{ij}$ is the directed, non-negative, weighted, connection between agents, with $a_{ij}=0$ if agents $i$ and $j$ are not neighbors, and $a_{ii} = 0$; and $\delta_i$ is the non-negative healing rate of agent $i$. In matrix form, with $p(t)$ representing the vector of the probabilities of infection of the agents, the model is
\begin{equation}\label{eq:van}
\dot{p}(t) = (BA - P(t)BA - D)p(t),
\end{equation}
where $B = diag(\beta_1,\dots, \beta_n)$, $A = [a_{ij}]$ represents the weighted network structure, $P(t)= diag(p_1(t),\dots, p_n(t))$, and $D = diag(\delta_1, \dots, \delta_n)$. 
Each node of the network can be interpreted as an individual agent \cite{van2009virus}, or as the centroid of a community, i.e.,  as a grouping of individuals \cite{fall2007epidemiological}.

In this paper, a time--varying extension of the model in \eqref{eq:van} is considered, that is 
\begin{equation}\label{eq:tv}
\dot{p}(t) = (BA(t) - P(t)BA(t) - D)p(t),
\end{equation}
where now $A(t)$ is a function of time. Note that $A(t)$ is not necessarily symmetric, and depicts the links between the $n$ agents, similar to an adjacency matrix without the constraint to be binary-valued. The mean field approximation in \eqref{eq:mean} remains unaffected by this extension as long as the assumption is made that the $a_{ij}(t)$'s are deterministic and known functions. We assume that $p_i(0) \geq 0$ for all $i=1,...,n$.

For completeness, we also include a time--varying extension of the model in \eqref{eq:2n}. The extended model is
\begin{equation}\label{eq:2nt}
    \frac{dy^T(t)}{dt} = y^T(t)Q(t),
\end{equation}
where
\begin{equation*}
q_{kl}(t)=
\begin{cases}
\delta, & \text{ if } x_i = 1, k = l + 2^{i-1}\\
\beta \displaystyle\sum_{j=1}^n a_{ij}(t)  x_j, &\text{ if } x_i = 0, k = l - 2^{i-1} \\[2.5ex]
-\displaystyle\sum_{j\neq l} q_{jl}(t), & \text{ if } k = l\\
0, & \text{ otherwise.}
\end{cases}
\end{equation*}
Note, due to the immense size of the $2^n$ model it is quite costly to employ, which is why the mean field approximation is useful.

\begin{lemma}\label{lem:pos}
If $p_i(0)\geq 0$, for all $i = 1,...,n$, then $p_i(t)\geq 0$ for all $t\geq0$, $i = 1,...,n$.
\end{lemma}
\begin{proof}
Assume $p_i(0) = 0$ and $p_j(0)\geq 0$ for all $j\neq i$. Then by \eqref{eq:vani}, $\dot{p}_i(0) \geq 0$, driving  $p_i(t)\geq 0$ for $t>0$. 

Assume $p_i(0) > 0$ and $p_j(0)\geq 0$ for all $j\neq i$. Since there exists a derivative by \eqref{eq:tv}, $p_i(t)$ is continuous. Now suppose $\dot{p}_i(t) < 0$, for some interval $0 \leq \tau \leq t \leq  T$, where, by the continuity of $p_i(t)$, at time $T$ we have $p_i(T) = 0$. Then, similar to the first part of the proof, by \eqref{eq:vani}, $\dot{p}_i(T) \geq 0$ and $p_i(t) \geq 0$ for $t> T$.
\end{proof}

\section{Stability Analysis of the Time--Varying Model} \label{sec:stab}

In this section, stability analysis of the disease free equilibrium for the time--varying model in \eqref{eq:tv} is performed for the deterministic and several stochastic cases.

\subsection{Deterministic Case}

The DFE is the state where $p_i(t) = 0 $ for all $i$, which from \eqref{eq:tv} implies $\dot{p}_i(t) = 0$  for all $i$. 
We show global exponential convergence to the DFE under certain conditions, to be made precise.

\subsubsection{Undirected Graph and Homogeneous in $\beta$ Case}
\begin{theorem}\label{thm:atorigin}
Suppose $\beta_i=\beta$ $\forall i$, 
 $A(t)$ is symmetric, piecewise continuous in $t$, and bounded, and $\sup_{t\geq0} \lambda_1(BA(t)-D)<0$. Then the DFE is globally exponentially stable (GES).
\end{theorem}
\begin{proof}
To simplify notation we will write $p=p(t)$. Consider an arbitrary $p\geq0$ and define a  Lyapunov function $V(p) = 
\frac{1}{2}p^T p$.  For $p\neq0$,
\begin{align*}
 \dot{V}(p) &= p^T\dot{p} = p^T(BA(t) - P(t)BA(t) - D)p  \\
 &\leq  p^T(BA(t) - D)p \\
 &\leq  \lambda_1(BA(t)-D)\|p\|^2  \\
 &\leq \left(\sup_{t\geq 0} \lambda_1(BA(t)-D)\right)\|p\|^2 <0.
\end{align*}
The first inequality holds because $(P(t)BA(t))_{ij}\geq 0, \ \forall i,j$ by our assumption that $\beta\geq 0$, $a_{ij}\geq 0$ for all $i,j$, and $p_i(t)\geq 0$ for all $i$, and by Lemma \ref{lem:pos}. The second inequality holds by the Rayleigh-Ritz Quotient (\cite{horn2012matrix}) because $BA(t)-D$ is symmetric (since $A$ is symmetric and $\beta_i=\beta$ $\forall i$). The last inequality holds by definition of the supremum. Therefore, since the system is piecewise continuous in $t$ and, by the boundedness of $A(t)$, locally Lipschitz in $p$ $\forall t,p \geq 0$, the system converges exponentially fast to the origin by Theorem 8.5 in \cite{khalil1996nonlinear}.
\end{proof}

\begin{remark}
Theorem \ref{thm:atorigin} requires $BA(t)-D$ to be symmetric, which considerably simplifies the analysis. The following lemma and theorem explore the case where symmetry is not assumed, that is, for the heterogeneous virus model (different $\beta_i$'s $\delta_i$'s $\forall i$) on directed graphs (non-symmetric $A$).
\end{remark}

\subsubsection{Directed Graph and Heterogeneous Virus Case} 

$ \ \ \ $
Consider the linearized system
\begin{equation}\label{eq:lin}
\dot{p}  = (BA(t) - D)p.
\end{equation}
The following result and proof are similar to Theorem 3.4.11 in \cite{ioannou2012robust} and the Lyapunov analysis is done point-wise in $t$.
\begin{definition}\label{def:gam}
Assume that for all $t\geq 0$, there exist finite $c(t),\lambda(t)>0$ such that
\begin{equation}\label{eq:c}
\|BA(t) - D\| \leq c(t) e^{-\lambda(t) t} \ \ \forall t\geq 0.
\end{equation}
We then define 
\begin{align}\label{eq:gam2}
\begin{split}
 \gamma_1 &:= \sup_{t\geq0}\int_0^{\infty} c(t)^2 e^{-2\lambda(t) \tau}d\tau \\
 &\geq \left\|\int_0^{\infty} e^{(BA(t) - D)^T\tau}e^{(BA(t) - D)\tau}d\tau\right\|.
\end{split}
\end{align}
\end{definition}

\begin{lemma}\label{lem:lin}
Consider the system in \eqref{eq:lin} with $A(t)$ continuously differentiable and $BA(t) - D$ bounded, that is, there exists an $L>0$ such that $\|BA(t) - D\|\leq L \ \forall t$. 
Assume that $\sup_{t\geq0} s_1(BA(t)-D)<0$, and $\gamma_1$ in Definition \ref{def:gam}, is well-defined and finite. If $\sup_{t\geq 0} \|B\dot{A}(t) - D\| < \frac{1}{2\gamma_1^2}$ or $\int_t^{t+T}\|B\dot{A}(s) - D\|ds \leq \mu T + \alpha$ for small enough $\mu>0$, then the DFE is globally exponentially stable.
\end{lemma}
\begin{proof}
Since $\sup_{t\geq0} s_1(BA(t)-D)<0$, we have $BA(t)-D$ is Hurwitz for all $t\geq 0$ and therefore \eqref{eq:lin} is exponentially stable for all $t\geq 0$.  
This also implies that for any given $t$, there exists a symmetric, positive definite $Q(t)$ (by Theorem 4.6 
of \cite{khalil1996nonlinear}) such that
\begin{equation}\label{eq:r}
Q(t)(BA(t) - D) + (BA(t) - D)^TQ(t) = -I.
\end{equation} 
Note that the solution to this equation is given by
\begin{equation}\label{eq:q}
Q(t) = \int_0^{\infty} e^{(BA(t) - D)^T\tau }e^{(BA(t) - D)\tau}d\tau.
\end{equation}
By our assumption, there exists an $L>0$ such that $\|BA(t) - D\|\leq L \ \forall t$, which implies
\begin{align*}
 \|p\| &= \Big\|e^{-(BA(t) - D)\tau}e^{(BA(t) - D)\tau}p\Big\|  \\
 &\leq    e^{\|BA(t) - D\|\tau}  \left\|e^{(BA(t) - D)\tau}p \right\| \\
 &\leq    e^{L\tau}  \left\|e^{(BA(t) - D)\tau}p \right\|,
\end{align*}
where the first inequality follows from the Cauchy-Schwartz inequality and some manipulation of the matrix exponential.
Therefore $\|e^{(BA(t) - D)\tau}p\| \geq e^{-L\tau} \|p\|$, so from \eqref{eq:q} we have 
\begin{equation}\label{eq:qlb}
\begin{split}
p^T Q(t)p 
 &= \int_0^{\infty} p^T e^{(BA(t) - D)^T\tau }e^{(BA(t) - D)\tau} p d\tau \\
 &=\int_0^{\infty} \|e^{(BA(t) - D)\tau} p\|^2 d\tau \\
 &\geq 
 \gamma_0 \|p\|^2, 
\end{split}
\end{equation}
where $\gamma_0 := \int_0^{\infty}  e^{-2L \tau}d\tau = \frac{1}{2L}$.

Consider $V(p,t) = p^T Q(t) p$. 
By \eqref{eq:qlb},  \eqref{eq:gam2} and \eqref{eq:q} we have
\begin{equation}\label{eq:ray}
\gamma_0 \|p\|^2 \leq p^T Q(t) p \leq \, \gamma_1\|p\|^2,
\end{equation}
where $\gamma_1$ is well-defined and finite by assumption.
Taking the time derivative of $V(p,t)$ gives
\begin{align}\label{eq:vdot}
\begin{split}
 \dot{V}
 &=  p^T( Q(t)(BA(t) - D) + (BA(t) - D)^TQ(t) + \dot{Q}(t)) p \\
 &= -\|p\|^2 + p^T \dot{Q}(t) p,
 \end{split}
\end{align}
where the second equality follows from \eqref{eq:r}. 
Taking the time derivative of \eqref{eq:r} gives,
\begin{equation*}
\dot{Q}(t)(BA(t) - D) + Q(t)(B\dot{A}(t) - D)  + (BA(t) - D)^T \dot{Q}(t) + (B\dot{A}(t) - D)^T Q(t) = 0, 
\end{equation*}
which implies
\begin{equation}\label{eq:der}
\dot{Q}(t)(BA(t) - D) + (BA(t) - D)^T \dot{Q}(t) = - Q(t)(B\dot{A}(t) - D) - (B\dot{A}(t) - D)^T Q(t) =:R(t).
\end{equation}
Note
\begin{equation}\label{eq:rnorm}
\|R(t)\| \leq 2\|Q(t)(B\dot{A}(t) - D)\| \leq 2\|Q(t)\| \|B\dot{A}(t) - D\|.
\end{equation}
The solution to \eqref{eq:der} is
\begin{equation*}
\dot{Q}(t) = \int_0^{\infty} e^{(BA(t) - D)^T\tau}R(t) e^{(BA(t) - D)\tau}d\tau.
\end{equation*}
Therefore 
\begin{align}\label{eq:gammas}
 \begin{split}
  \|\dot{Q}(t)\| &\leq \|R(t)\| \left\|\int_0^{\infty} e^{(BA(t) - D)^T\tau} e^{(BA(t) - D)\tau}d\tau\right\|  \\
  &\leq    \gamma_1 \|R(t)\|  \\ 
  &\leq    \gamma_1 (2\|Q(t)\| \|B\dot{A}(t)-D\|) \\
  &\leq    2\gamma_1^2\|B\dot{A}(t) - D\|,
 \end{split}
\end{align}
where the second and last inequalities hold by \eqref{eq:gam2} and the third holds by \eqref{eq:rnorm}.
Substituting \eqref{eq:gammas} into \eqref{eq:vdot} and using \eqref{eq:r} gives
\begin{align}\label{eq:vdot2}
 \begin{split}
  \dot{V}(p,t) &\leq -\|p\|^2 + 2\gamma_1^2\|B\dot{A}(t) - D\| \|p\|^2 \\ 
  &= -(1-  2\gamma_1^2\|B\dot{A}(t) - D\|) \|p\|^2.
 \end{split}
\end{align}
Thus for  $\sup_{t>0} \|B\dot{A}(t) - D\| < \frac{1}{2\gamma_1^2}$,  the origin is GES.

Otherwise, using 
\eqref{eq:ray}, we can rewrite \eqref{eq:vdot2} as
\begin{align*}
  \dot{V}(p,t) &\leq -\frac{1}{\gamma_1}V(p,t) + 2\frac{\gamma_1^2}{\gamma_0}\|B\dot{A}(t) - D\|V(p,t) \\ 
  &= -\left( \frac{1}{\gamma_1} - 2\frac{\gamma_1^2}{\gamma_0}\|B\dot{A}(t) - D\|\right) V(p,t).
\end{align*}
By the Comparison Principle (See e.g. Section 3.4 in \cite{khalil1996nonlinear}), we have
\begin{align*}
  V(t) &\leq e ^{-\int_{t_0}^t (\frac{1}{\gamma_1} - 2\frac{\gamma_1^2}{\gamma_0}\|B\dot{A}(t) - D\|)ds}V(t_0) \\ 
  &\leq e ^{2\frac{\gamma_1^2}{\gamma_0}\alpha}e^{-(\frac{1}{\gamma_1} - 2\frac{\gamma_1^2}{\gamma_0}\mu)(t-t_0)}V(t_0).
\end{align*}
Therefore, since $\int_t^{t+T}\|B\dot{A}(s) - D\|ds \leq \mu T + \alpha$ by our assumption, if  $\mu<\frac{\gamma_0}{2\gamma_1^3}$ then, with $\bar{c} = e ^{2\frac{\gamma_1^2}{\gamma_0}\alpha}$ and $\bar{\lambda} = \frac{1}{\gamma_1} - 2\frac{\gamma_1^2}{\gamma_0}\mu$, 
\begin{equation*}
\|p(t)\| \leq \bar{c}e^{-\bar{\lambda}(t-t_0)}\|p(t_0)\|,
\end{equation*}
that is, the origin is globally exponentially stable.
\end{proof}

\begin{theorem}\label{thm:DFEasymp}
Consider the system in \eqref{eq:tv} with $A(t)$ continuously differentiable and $BA(t) - D$ bounded, that is, there exists an $L>0$ such that $\|BA(t) - D\|\leq L \ \forall t$. 
Assume that $\sup_{t\geq0} s_1(BA(t)-D)<0$, and $\gamma_1$ in Definition \ref{def:gam}, is well-defined and finite. If $\sup_{t\geq 0} \|B\dot{A}(t) - D\| < \frac{1}{2\gamma_1^2}$ or $\int_t^{t+T}\|B\dot{A}(s) - D\|ds \leq \mu T + \alpha$ for small enough $\mu>0$, then the DFE is globally exponentially stable.
\end{theorem}
\begin{proof}
Note that since $(P(t)BA(t))_{ij}\geq 0 \ \forall i,j$, by construction and Lemma \ref{lem:pos},
\begin{align*}
\dot{p} &= (BA(t) - P(t)BA(t) - D)p  \\
 &\leq  (BA(t) - D)p.
\end{align*}
Therefore, by Gr\"{o}nwall's Inequality (\cite{gronwall1919note}), the solution of the original system will be bounded above by the solution of the linear system. 
Thus by Lemma \ref{lem:lin}, the DFE is GES for the system in \eqref{eq:tv}.
\end{proof}

\subsection{Stochastic Model}

In this section we explore introducing randomness 
using two different models, a generic additive noise model and an Ito's formula-type model. 
Note that  the mean-field step in \eqref{eq:mean}, in essence, removes the randomness that was included in the original $2^n$ model. Therefore, an exploration of random graph structures, while interesting as an extension of the $n$-intertwined Markov model in \eqref{eq:tv}, does not accurately approximate the  $2^n$ model with random graph structure. However, the data coming in could be thought of as a signal, 
with additive noise, which is reflected in the models presented in this section. 
\subsubsection{Generic Noise}
Consider the system
\begin{equation}\label{eq:noise}
\dot{p}(t) = \underbrace{(BA(t) - P(t)BA(t) - D)p(t)}_{F(t,p)} + g(t,p)\xi(t,\omega),
\end{equation}
which represents a perturbation to the model in \eqref{eq:tv}, where $\xi(t,\omega)\in \mathbb{R}^{k}$ is a zero mean, measurable stochastic process, $A(t)$, $B$,  $D$, and $g(t,p)$ are deterministic, $g(t,p) \in \mathbb{R}^{n \times k}$, and $g(t,0) = 0$ for all $t$. 
 We assume that $p_i(0) \geq 0$ for all $i=1,...,n$.
\begin{lemma}\label{lem:posSto}
Consider the system in \eqref{eq:noise} with $p_i(0)\geq 0$, for all $i=1,...,n$. If $\xi(t,\omega)\in \mathbb{R}^{k}$ is a zero mean, measurable stochastic process and, for all $i = 1,...,n$ there exists a $k_i>0$ such that   $\|g_i(t,p)\| \leq k_i |p_i|^2$ for all $t\geq 0$, then $p_i(t)\geq 0$ for all $t\geq0$, $i=1,...,n$.
\end{lemma}
\begin{proof}
By Lemma \ref{lem:pos}, the deterministic part of \eqref{eq:noise}, $F(t,p)$, is non-negative for all $t\geq 0$. Therefore we turn our attention to the $g_i(t,p)\xi(t,\omega)$ term, where $i$ refers to the $i$th row of $g(t,p)$. By the zero mean and independence assumptions, $\xi(t,\omega)$ can be negative for any $t\geq 0$. However, by our assumption that $\|g_i(t,p)\| \leq k_i |p_i|^2$ for all $t$, we have
, for any $t\geq 0$,
\begin{align*}
    \lim_{p_i \rightarrow 0} \|g_i(t,p)\xi(t,\omega)\| 
    &\leq \lim_{p \rightarrow 0}\|g_i(t,p)\|  \|\xi(t,\omega)\| \\
    & \leq \lim_{p_i \rightarrow 0} k_i |p_i|^2 \|\xi(t,\omega)\| \\
    & = 0.
\end{align*}
Therefore if $p_i(0)\geq 0$ then as $p_i$ approaches zero, the random part of the derivative vanishes. Therefore if $p_i(t) = 0$ then $\dot{p}_i(t)\geq 0$, and consequently $p_i(t)\geq 0$ for all $t\geq 0$. 
\end{proof}

We have the following result: 
\begin{theorem} \label{thm:noise}
Consider the system in \eqref{eq:noise} with $\beta_i=\beta$ $\forall i$, $A(t)$  symmetric, piecewise continuous in $t$, and bounded, and  $\sup_{t\geq0} \lambda_1(BA(t)-D)<0$. 
If $\xi(t,\omega)\in \mathbb{R}^{k}$ is a zero mean, measurable stochastic process and, for all $i = 1,...,n$ there exists a $k_i>0$ such that   $\|g_i(t,p)\| \leq k_i |p_i|^2$ for all $t\geq 0$,  then the DFE is globally exponentially stable in expectation.
\end{theorem}
\begin{proof}
Consider an arbitrary $p$ and define a  Lyapunov function $V(p) = 
\frac{1}{2}p^T p$.  For $p\neq0$,
\begin{align}
 E[\dot{V}(p)|p] &= E[p^T\dot{p}|p] \nonumber \\ 
 &= E[p^T(BA(t) - P(t)BA(t) - D)p   + p^Tg(t,p)\xi(t,\omega) |p] \nonumber \\
 &= p^T(BA(t) - P(t)BA(t) - D)p   + p^Tg(t,p)E[\xi(t,\omega)] \label{eq:thm33} \\
 &=  p^T(BA(t) - P(t)BA(t) - D)p \label{eq:thm34} \\
 &\leq  p^T(BA(t) - D)p \label{eq:thm35} \\
 &\leq  \lambda_1(BA(t)-D)\|p\|^2 \nonumber \\
 &\leq \left(\sup_{t\geq 0} \lambda_1(BA(t)-D)\right)\|p\|^2 <0, \nonumber
\end{align}
where \eqref{eq:thm33} and \eqref{eq:thm34} hold because, by assumption, $E[\xi(t,\omega)|p] = E[\xi(t,\omega)] = 0$, and \eqref{eq:thm35} holds by our assumption that $p_i(0) \geq 0$ for all $i=1,...,n$ and by Lemma \ref{lem:posSto}.
Thus, since the system is piecewise continuous in $t$ and locally Lipschitz in $p$ $\forall t,p \geq 0$, the system converges exponentially fast to the origin by Theorem 8.5 in \cite{khalil1996nonlinear} in expectation.
\end{proof}

We will use the result stated in Lemma \ref{lem:khas}, from \cite{khasminskii2011stochastic}, to prove Theorem \ref{thm:noise2}. Note that $d^0V/dt$ is defined as
\begin{equation}
\frac{d^0V}{dt} := \frac{\partial V}{\partial t} + \sum_{i=1}^n \frac{\partial V}{\partial p_i} F_i(t,p),
\end{equation}
where $F_i(t,p)$ is the $i$th entry of $F(t,p)$ as defined in \eqref{eq:noise}.
\begin{lemma}\label{lem:khas} \em{(Theorem 1.12, 
in \cite{khasminskii2011stochastic})}
Consider the system \eqref{eq:noise} with a Lyapunov function $V(p,t)$ that is positive definite uniformly in $t$ and $V(0,t) = 0$. If $\xi(t,\omega)$ satisfies the strong law of large numbers, 
\begin{gather}
\sup_{t\geq 0} E|\xi(t,\omega)| < \frac{c_1}{bc_2},\label{eq:exp}\\
\frac{d^0V}{dt} \leq -c_1 V, \text{ and } \|g\| \leq c_2 V,\label{eq:Vbound}
\end{gather}
for some constants $c_1,c_2,b>0$, then the origin is almost surely asymptotically stable.
\end{lemma}

\begin{theorem} \label{thm:noise2}
Consider the system in \eqref{eq:noise} with $\beta_i=\beta$ $\forall i$, $A(t)$  symmetric, and  $ \sup_{t\geq 0} \lambda_1(BA(t)-D) <0$. If $\xi(t,\omega)$ is a zero mean, independent, identically distributed (i.i.d.), measurable stochastic process, and for all $i = 1,...,n$ there exists a $k_i>0$ such that   $\|g_i(t,p)\| \leq k_i |p_i|^2$ for all $t\geq 0$, 
then the origin is almost surely asymptotically stable.
\end{theorem}
\begin{proof}

Consider the Lyapunov function candidate $V(p) = \frac{1}{2}p^T p$. Clearly $V$ is  positive definite uniformly in $t$. Since $ \sup_{t\geq 0} \lambda_1(BA(t)-D) <0$, there exists $\delta >0$ such that for all $t\geq 0$, $ \sup_{t\geq 0} \lambda_1(BA(t)-D) \leq -\delta$. Therefore, for $p\geq0$,
\begin{align}
\frac{d^0V}{dt} &= p^T(BA(t) - P(t)BA(t) - D)p  \nonumber\\
 &\leq  p^T(BA(t) - D)p \label{eq:thm41}\\
 &\leq \left(\sup_{t\geq 0} \lambda_1(BA(t)-D)\right)\|p\|^2 \nonumber\\
 &\leq -\delta p^Tp = {-2\delta}{V},\nonumber
\end{align}
where \eqref{eq:thm41} follows from our assumption that $p_i(0) \geq 0$ for all $i=1,...,n$ and by Lemma \ref{lem:posSto}.


By our assumption $\|g_i(t,p)\| \leq k_i |p_i|^2$ we have that 
\begin{align}
    \|g(t,p)\|^2 &\leq \sum_{i=1}^n\|g_i(t,p)\|^2 \label{eq:g1} \\
    &\leq \sum_{i=1}^n (k_i|p_i|^2)^2 \nonumber\\
    &\leq c\sum_{i=1}^n (|p_i|^2)^2 \nonumber \\
    &\leq c\left(\sum_{i=1}^n |p_i|^2\right)^2 \label{eq:g4} \\
    &\leq c\|p\|^4,\label{eq:gend}
\end{align}
for all $t\geq 0$, where $c = n(\max_i k_i^2)$. Note \eqref{eq:g1} holds by the relationship between the 2-induced norm and the Frobenius norm and \eqref{eq:g4} holds because the cross terms $|p_i|^2|p_j|^2\geq0$.
Therefore $\|g(t,p)\|\leq \sqrt{c}V$; so \eqref{eq:Vbound} is satisfied.
Also, since $\xi(t,\omega)$ is i.i.d. it satisfies the strong law of large numbers and \eqref{eq:exp} is satisfied by the zero mean assumption. Therefore by Lemma \ref{lem:khas}, the origin is almost surely asymptotically stable.
\end{proof}

A similar result can be shown for the directed graph case. 

\begin{theorem} \label{thm:noise3}
Consider the system in \eqref{eq:noise} with $A(t)$ continuously differentiable and $BA(t) - D$ bounded, that is, there exists an $L>0$ such that $\|BA(t) - D\|\leq L \ \forall t$. Further suppose $\xi(t,\omega)$ are zero mean and i.i.d., and for all $i = 1,...,n$ there exists a $k_i>0$ such that   $\|g_i(t,p)\| \leq k_i |p_i|^2$ for all $t\geq 0$. 
Assume that $\sup_{t\geq0} s_1(BA(t)-D)<0$, and $\gamma_1$ in Definition \ref{def:gam}, is well-defined and finite. If $\sup_{t\geq 0} \|B\dot{A}(t) - D\| < \frac{1}{2\gamma_1^2}$ or $\int_t^{t+T}\|B\dot{A}(s) - D\|ds \leq \mu T + \alpha$ for small enough $\mu>0$, 
then the origin is almost surely asymptotically stable.
\end{theorem}
\begin{proof}

Similar to the proof of Theorems \ref{thm:DFEasymp} and \ref{thm:noise2}, appealing to Lemma \ref{lem:lin}, it can be shown $\frac{d^0V}{dt} \leq -c_1 V$.



By our assumption $\|g_i(t,p)\| \leq k_i |p_i|^2$ and \eqref{eq:g1}-\eqref{eq:gend}, we have that $\|g(t,p)\| \leq  \sqrt{c}V$ for all $t\geq 0$, where $c = n(\max_i k_i^2)$.
Also, since $\xi(t,\omega)$ is i.i.d., it satisfies the strong law of large numbers and \eqref{eq:exp} is satisfied by the zero mean assumption. Therefore, by Lemma \ref{lem:khas}, the origin is almost surely asymptotically stable.
\end{proof}

\subsubsection{Ito's formula-based modeling}  
Consider the system
\begin{equation}\label{eq:ito}
d{p}(t) = \underbrace{(BA(t) - P(t)BA(t) - D)p(t)dt}_{F(t,p)} + g(t,p)dw,
\end{equation}
which represents a perturbation to the model in \eqref{eq:tv}, where $w$ is a $d$-dimensional vector of independent standard Wiener processes and $A(t)$, $B$,  $D$, and $g(t,p)$ are deterministic. Again assume that $g(t,0) = 0$ for all $t$, and $p_i(0)\geq 0$ for all $i=1,...,n$.

Similar to Lemmas \ref{lem:pos} and \ref{lem:posSto}, we can state a positivity result for $p(t)$ in \eqref{eq:ito}:
\begin{lemma}\label{lem:posIto}
Consider the system in \eqref{eq:ito} with $p_i(0)\geq 0$ for all $i = 1,...,n$. If, for all $i = 1,...,n$ there exists a $k_i>0$ such that   $\|g_i(t,p)\| \leq k_i |p_i|$ for all $t\geq 0$, then $p_i(t)\geq 0$ for all $t\geq0$, $i = 1,...,n$.
\end{lemma}
\begin{proof}
By Lemma \ref{lem:pos} the deterministic part of \eqref{eq:ito}, $F(t,p)$, is non-negative for all $t\geq 0$. Therefore we turn our attention to the $g_i(t,p)\xi(t,\omega)$ term, where $i$ refers to the $i$th row of $g(t,p)$. By our assumption that $\|g_i(t,p)\| \leq k_i |p_i|$ we have
, for any $t\geq 0$,
\begin{align*}
    \lim_{p_i \rightarrow 0} \|g_i(t,p)dw\| 
    &\leq \lim_{p \rightarrow 0}\|g_i(t,p)\|  \|dw\| \\
    & \leq \lim_{p_i \rightarrow 0} k_i |p_i| \|dw\| \\
    & = 0.
\end{align*}
Therefore if $p_i(0)\geq 0$ then as $p_i$ approaches zero, the random part of the derivative vanishes. Further if $p_i(t) = 0$, then $\dot{p}_i(t)\geq 0$, and consequently $p_i(t)\geq 0$ for all $t\geq 0$. 
\end{proof}

We now call several results from \cite{khasminskii2011stochastic}. Consider the system 
\begin{equation}\label{eq:gen}
dp = b(t,p)dt + \sum_{r=1}^k \sigma_r(t,p) dw_r(t),
\end{equation}
where $w_i$'s are independent standard Wiener processes and $p(t)$, $b(t,p)$, and  $\sigma_r(t,p)$ are vectors in $\mathbb{R}^d$. The {\em generator operator} (see Chapter 5 in \cite{khasminskii2011stochastic}), which generalizes the operation of differentiating a Lyapunov function $V$, is given by
\begin{equation}
\mathcal{L} =  \frac{\partial}{\partial t} +\langle b,\frac{\partial}{\partial p}\rangle +\sum_{r} \langle \sigma_r,\frac{\partial}{\partial p} \rangle^2_,
\end{equation}
where $<\cdot,\cdot>$ is the inner product and $\frac{\partial}{\partial p} = \left[ \frac{\partial}{\partial p_1}, \dots, \frac{\partial}{\partial p_n}\right]^T_.$
\begin{definition} \em{(Section 5.7 
in \cite{khasminskii2011stochastic})}
A system is {\em exponentially 2-stable} if for some constants $a,b$, and $\forall t\geq 0$,
\begin{equation*}
E\|p(t)\|^2 \leq a\|p(0)\|^2e^{bt}.
\end{equation*}
\end{definition}
\begin{theorem}\em{(Theorem 5.11,15, Section 5.7 
in \cite{khasminskii2011stochastic})} \label{thm:exp2}
Given a system as in Equation \eqref{eq:gen}, if there exists a, twice continuously differentiable with respect to $p$ and continuously differentiable with respect to $t$, Lyapunov function $V(t,p)$ such that
\begin{equation}\label{eq:vb}
k_1 \|p\|^2\leq V(t,p) \leq k_2 \|p\|^2,
\end{equation}
\begin{equation}\label{eq:lv}
\mathcal{L}V(t,p) \leq -k_3 \|p\|^2,
\end{equation}
for some positive constants $k_1, k_2, k_3$, then the origin is exponentially 2-stable. Furthermore, the origin is almost surely exponentially stable.
\end{theorem}

\begin{theorem} \label{thm:ito2}
Consider the system in \eqref{eq:ito} with $g(t,p)$ bounded and locally Lipschitz in p(t) uniformly in $t$, $w$ are independent standard Wiener processes, and for all $i = 1,...,n$ there exists a $k_i>0$ such that   $\|g_i(t,p)\| \leq k_i |p_i|$ for all $t\geq 0$. If 
    $\beta_i=\beta$ $\forall i$, 
$A(t)$ is symmetric, piecewise continuous in $t$, and bounded, and $\sup_{t\geq0} \lambda_1(BA(t)-D)<-c$, with $c := \sum_{i=1}^n k_i^2+\epsilon$, $\epsilon>0$, 
then the origin is exponentially 2-stable and almost surely exponentially stable.
\end{theorem}
\begin{proof}
Consider the Lyapunov function candidate $V(p) = \frac{1}{2}p^T p$. Clearly Equation \eqref{eq:vb} is satisfied. 
Since 
$\frac{\partial V}{\partial p} = \nabla V = p$ and  $\nabla^2 V =I$, we have
\begin{align}
\mathcal{L}V(p) &= \langle b(t,p),\frac{\partial V}{\partial p} \rangle + \langle g(t,p),\frac{\partial V}{\partial p} \rangle^2 \nonumber \\
&= p^Tb(t,p) +\frac{1}{2}\sum_{i,j} (g(t,p)g^T(t,p))_{ij}\frac{\partial V}{\partial p_i\partial p_j} \nonumber \\
&= p^T(BA(t) - PBA(t) - D)p +\frac{1}{2}\sum_{i=1}^n g_i(t,p)g_i^T(t,p) \nonumber \\
&\leq p^T(BA(t) - D)p +\frac{1}{2}\sum_{i=1}^n (k_i |p_i|)^2 \label{eq:lvp}\\
 &\leq \left(\sup_{t\geq 0} \lambda_1(BA(t)-D) + \frac{1}{2}\sum_{i=1}^n k_i^2\right)\|p\|^2 \label{eq:lvp2}\\
 &< \left(-c+ \frac{1}{2}\sum_{i=1}^n k_i^2\right)\|p\|^2 \nonumber \\
 &= -\epsilon \|p\|^2,\label{eq:epsp}
\end{align}
where \eqref{eq:lvp} holds because $PBA(t) \geq 0$, by construction and Lemma \ref{lem:posIto} and  $\|g_i(t,p)\| \leq k_i |p_i|$, for all $i$, by assumption; \eqref{eq:lvp2} holds by the symmetry of $BA(t)$; and \eqref{eq:epsp} holds by definition of $c$. Thus \eqref{eq:lv} is satisfied 
and therefore by Theorem \ref{thm:exp2}, the origin is exponentially 2-stable and almost surely exponentially stable.
\end{proof}

\section{Simulations and Extensions}\label{sec:sim}

In this section we first compare via simulations the models in \eqref{eq:2n} and \eqref{eq:p} over different graph structures, and then compare them for the complete, fully connected graph with \eqref{eq:sis}. We also provide a comparison of the time--varying graph structure extensions of the models in \eqref{eq:tv} and \eqref{eq:2nt}. 
A variety of time--varying simulations are presented, leading to several corollaries and remarks.
Since the infection rate $p(t)$ and the location of the states, $z(t)$, are both time dependent, the simulations are best viewed in video format with links provided in the captions of the tables and figures. 

\subsection{Comparison: $n$-intertwined and $2^n$ Markov Chain Models}

While an initial analysis has been completed in \cite{van2009virus}, to evaluate the accuracy of the mean field approximation used in the  derivation of the $n$-intertwined model in \eqref{eq:p}, 
we  further the analysis here. We simulate various  graph structures for both static and dynamic cases, for both the $n$-intertwined model  and the $2^n$ model.

\begin{figure}
    \centering
      \includegraphics[width=\columnwidth]{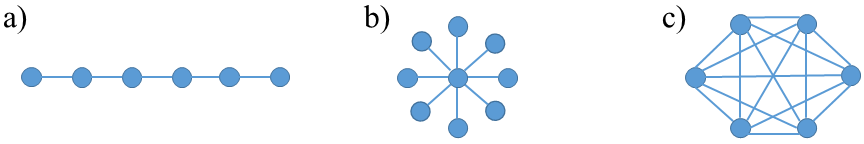}
      \caption{Graphs: a) Line b) Star c) Complete.}
\label{fig:graphs}
\end{figure} 
\begin{table}[]
\begin{center}
    \begin{tabular}{| l | l | l | l || l | l | l || l | l | l |}
    \hline
    $\frac{\beta }{\delta}$ & $\frac{1}{10}$ & $1$ & $10$ & $\frac{1}{10}$ & $1$ & $10$ & $\frac{1}{10}$ & $1$ & $10$ \\ \hline
    $n$ & \multicolumn{3}{ l|| }{ \ \ \ \ \ \  $p^1(0)$}  & \multicolumn{3}{ l|| }{ \ \ \ \ \ \ $p^2(0)$}& \multicolumn{3}{ l| }{ \ \ \ \ \ \ $p^3(0)$} \\ \hline
    6 & 0  &  1.02  &  0.64 & 0  &  1.02  &  0.45 & 0  &  1.02  &  0.30 \\ \hline
    8 & 0  &  1.24  &  0.31 & 0  &  1.24  &  0.04 & 0  &  1.24  &  0.30 \\ \hline
    10 & 0  &  1.43  &  0.33 & 0  &  1.43  &  0.02 & 0  &  1.43  &  0.33 \\ \hline
    13 & 0  &  1.67  &  0.37 & 0  &  1.67  &  0.02 & 0  &  1.67  &  0.37 \\ 
    \hline
    \end{tabular}
    \caption{$\|v(T)-p(T)\|$ for the line graph, $T=10000$.  For a simulation of $n=6$, $\frac{\beta }{\delta}=1$, and $p(0)=p^3(0) = [ 1 \ 0 \ \cdots \ 0]^T$ see \href{https://youtu.be/E49OTI4Pgh0}{youtu.be/E49OTI4Pgh0}.}
    \label{tab:line}
\end{center}
\end{table}


\begin{table}[]
\begin{center}
    \begin{tabular}{| l | l | l | l || l | l | l || l | l | l |}
    \hline
    $\frac{\beta }{\delta}$ & $\frac{1}{10}$ & $1$ & $10$ & $\frac{1}{10}$ & $1$ & $10$ & $\frac{1}{10}$ & $1$ & $10$ \\ \hline
    $n$ & \multicolumn{3}{ l|| }{ \ \ \ \ \  \ $p^1(0)$}  & \multicolumn{3}{ l|| }{ \ \ \ \ \ \ $p^2(0)$}& \multicolumn{3}{ l| }{ \ \  \ \ \ \ $p^3(0)$} \\ \hline
    6 & 0  &  1.12  & 0.35 & 0  &  1.12  & 0.35 &  0  &  1.12  &  0.39 \\ \hline
    8 & 0  &  1.36  & 0.01 & 0  &  1.36  & 0.02 &  0  &  1.36  &  0.06 \\ \hline
    10 & 0  &  1.55  & 0.00 & 0  &  1.55 & 0.00 &  0  &  1.55  &  0.04 \\ \hline
    13 & 0  &  1.80  &  0  & 0  &  1.80  &  0 &  0  &  1.80  &  0.03   \\
    \hline
    \end{tabular}
    \caption{$\|v(T)-p(T)\|$ for the star graph, $T=10000$. For a simulation of $n=6$, $\frac{\beta }{\delta}=1$, and $p(0)=p^3(0) = [ 1 \ 0 \ \cdots\ 0]^T$ see \href{https://youtu.be/XOdNUDFngO4}{youtu.be/XOdNUDFngO4}.}
    \label{tab:star}
\end{center}
\end{table}
\subsubsection{Static Graphs} The static graph structures considered in the simulations are line graphs, star (hub--spoke) graphs, and complete graphs; see Figure \ref{fig:graphs} for examples of each graph structure. All the $A$ matrices for these graphs are symmetric and binary-valued. In the star graph, the central node is the first agent. 
Each simulation was run for 10,000 time steps  (final time $T=10000$), with three initial conditions: 1) every agent infected, $p^1(0) = [ 1 \ \cdots\ 1]^T$, 2) half the agents infected, $p^2(0) = [ 1 \ \cdots\ 1 \ 0 \ \cdots\ 0]^T$, and 3) one agent infected, $p^3(0) = [ 1 \ 0 \ \cdots\ 0]^T$. We  explore the homogeneous virus case in these tests. The $(\beta,\delta)$ pairs were $[ (.1,1), (.5,.5), (1, .1)]$, and the number of agents, $n = 6,8,10, 13$. We limited simulations to these $n$ values since mean field approximations are typically worse for small values of $n$ and there is a computational limitation due to the size of the $2^n$ model.

The results are given in Tables \ref{tab:line}-\ref{tab:full} in terms of the 2-norm of the difference between the state of the $n$-intertwined Markov chain model at the final time ($p(T)$), and the mean of the $2^n$ Markov model at the final time ($v(T)$ as defined by \eqref{eq:v}). 
We round the error to zero if it is less than $0.001$. Since  the $n$-intertwined Markov chain model is an upper bounding approximation, the results show that the two models converge to the DFE for $\beta/\delta = 1/10$, resulting in small errors. However for $\beta/\delta = 1$, the models differ quite drastically; the $n$-intertwined Markov chain model appears to be at a NDFE while the $2^n$ model appears, in most cases, to be at or close to the DFE, resulting in large errors. For $\beta/\delta = 10$, the $n$-intertwined Markov chain model again performs quite well since both models are at a NDFE. However, not as well as for $\beta/\delta = 1/10$, since the models are at different NDFEs. Note also, for most of the $\beta/\delta = 10$ cases and for $\beta/\delta = 1$ for the complete graph, the errors decrease as $n$ increases.

\begin{table}[]
\begin{center}
    \begin{tabular}{| l | l | l | l | l | l | l | l | l | l |}
    \hline
    $\frac{\beta }{\delta}$ & $\frac{1}{10}$ & $1$ & $10$ & $\frac{1}{10}$ & $1$ & $10$ & $\frac{1}{10}$ & $1$ & $10$ \\ \hline
    $n$ & \multicolumn{3}{ l| }{ \ \ \ \ \  \ $p^1(0)$}  & \multicolumn{3}{ l| }{ \ \ \ \ \ \ $p^2(0)$}& \multicolumn{3}{ l| }{ \ \  \ \ \ \ $p^3(0)$} \\ \hline
    6  & 0    & 1.96  & 0.0 &  0   & 1.96  & 0.0 &  0   &  1.96  &  0.05 \\ \hline
    8  & 0    & 2.14  &  0   &  0   & 2.14  &  0   &  0   &  2.18  &  0.04 \\ \hline
    10 & 0    & 0.12  &  0   &  0   & 0.12  &  0   &  0   &  0.42  &  0.03 \\ \hline
    13 & 0.6  & 0.00  & 0.0 & 0.6  & 0.00  &  0   & 0.6  &  0.28  &  0.03 \\
    \hline
    \end{tabular}
    \caption{$\|v(T)-p(T)\|$ for the complete graph, $T=10000$. For a simulation of $n=6$, $\frac{\beta }{\delta}=1$, and $p^3(0)$ see \href{https://youtu.be/VTFZDdXsC6M}{youtu.be/VTFZDdXsC6M}.}
    \label{tab:full}
\end{center}
\end{table}


\subsubsection{Comparison of the Complete Graph Models}

For completeness we include a comparison of the results from the static complete graph in Table \ref{tab:full} to the results of simulating the original model in \eqref{eq:sis}. Since \eqref{eq:sis} models the population as two groups, we sum the results of the mean field ($\displaystyle\sum_{i=1}^{n} p_i(T)$) and the means of the $2^n$ model ($\displaystyle \sum_{i=1}^{n} v_i(T)$, with $v(t)$ defined in \eqref{eq:v}. These sums, with the results of \eqref{eq:sis} ($I(T)$), are compared in Tables \ref{tab:kermack1}-\ref{tab:kermack2}. All three models perform very similarly except for the $2^n$ model for $(n, \frac{\beta }{ \delta}) = \{(13,1/10), (6,1), (8,1)\}$, which is consistent with Table \ref{tab:full}. 
\begin{table}[]
\begin{center}
    \begin{tabular}{| l | l | l | l |}
    \hline
    $n$, $\frac{\beta }{ \delta}$ & $1/10$ & $1$ & $10$ \\ \hline
    6  & 0, 0, 0  &   5.0, 4.8, 0  &   5.90, 5.88, 5.88 \\ \hline
    8  & 0, 0, 0  &   7.0, 6.86, .82  &   7.90, 7.89, 7.89 \\  \hline
    10 & 0, 0, 0  &   9.00, 8.89, 8.52  &   9.90, 9.89, 9.89 \\ \hline
    13 & 3.0, 2.17, 0  &  12.0, 11.9, 11.9  &  12.90, 12.89, 12.89  \\
    \hline
    \end{tabular}
    \caption{Comparison of complete graph models to \eqref{eq:sis}: Each cell is $I(T), \sum_{i=1}^{n} p_i(T), \sum_{i=1}^{2^n} v_i(T)$ with $T=10000$ and $I(0) = n$.}
    \label{tab:kermack1}
\end{center}
\end{table}

\begin{table}[]
\begin{center}
    \begin{tabular}{| l | l | l | l |}
    \hline
    $n$, $\frac{\beta }{ \delta}$ & $1/10$ & $1$ & $10$ \\ \hline
    6  &  0, 0, 0  &   5.0, 4.8, 0 &   5.90, 5.88, 5.88 \\ \hline
    8  &  0, 0, 0  &   7.0, 6.9, .82  &   7.90, 7.89, 7.89 \\  \hline
    10 &  0, 0, 0  &   9.00, 8.89, 8.52  &   9.90, 9.89, 9.89 \\ \hline
    13 &  3.0, 2.17,0  &  12.0, 11.9, 11.9  &  12.90, 12.89, 12.89 \\
    \hline
    \end{tabular}
    \caption{Comparison of complete graph models to \eqref{eq:sis}: Each cell is $I(T), \sum_{i=1}^{n} p_i(T), \sum_{i=1}^{2^n} v_i(T)$, with $T=10000$ and $I(0) = floor(n/2)$.}
    \label{tab:kermack1.5}
\end{center}
\end{table}

\begin{table}[]
\begin{center}
    \begin{tabular}{| l | l | l | l |}
    \hline
    $n$, $\frac{\beta }{ \delta}$ & $1/10$ & $1$ & $10$ \\ \hline
    6  & 0, 0, 0  &   5.0, 4.8, 0  &   5.90, 5.88, 5.76 \\ \hline
    8  & 0, 0, 0  &   7.0, 5.5, .70  &   7.90, 7.89, 7.77 \\  \hline
    10 & 0, 0, 0  &   9.00, 8.89, 7.56  &   9.90, 9.89, 9.78 \\ \hline
    13 & 3.0, 2.17, 0  &  12.0, 11.9, 10.9  &  12.90, 12.89, 12.78  \\
    \hline
    \end{tabular}
    \caption{Comparison of complete graph models to \eqref{eq:sis}: Each cell is $I(T), \sum_{i=1}^{n} p_i(T), \sum_{i=1}^{2^n} v_i(T)$, with $T=10000$ and $I(0) = 1$.}
    \label{tab:kermack2}
\end{center}
\end{table}

\subsubsection{Dynamic Graphs}

In this section we use several examples to highlight the effectiveness and ineffectiveness of the $n$-intertwined model in \eqref{eq:tv} as a mean field approximation of \eqref{eq:2nt} for dynamic graph structures. 
For these simulations, the weighting matrix $A(t)$ is dependent on the agents' relative positions; that is, using the definition from \cite{shi2000normalized}, for some radius $r$ and $i\neq j$,
\begin{equation}\label{eq:at}
a_{ij}(t) = \begin{cases}
    e^{-\|z_i(t) - z_j(t)\|^2}, & \text{if } \|z_i(t) - z_j(t)\| < r \\
    0,              & \text{otherwise},
\end{cases}
\end{equation}
where $z_i(t)\in \mathbb{R}^d$ is the position of agent $i$ in $d$-space. 
Note that under the construction in \eqref{eq:at}, $A(t)$ is undirected.

First consider the case of constant drift for the positional dynamics of the agents, that is,
\begin{equation}\label{eq:phi}
\dot{z}(t) = \phi,
\end{equation}
where $\phi$ is some constant vector. 
As we see in Figure \ref{fig:compDrift}, the upper bounding nature of the $n$-intertwined Markov chain model leads to a decent approximation; at time step 40, the $2^n$ model has reached the DFE, whereas the $n$-intertwined Markov chain model has not. However,  the $n$-intertwined  model reaches the DFE shortly thereafter 
(see the link referenced  in the caption of Figure \ref{fig:compDrift}). 

\begin{figure}
    \centering
      \includegraphics[width=.95\columnwidth]{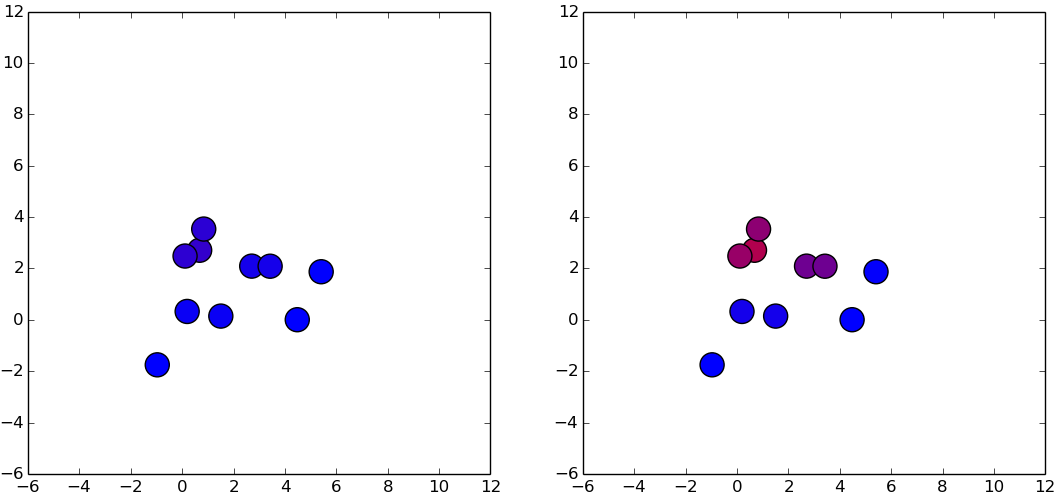}
      \caption{This system has constant drift as explained in \eqref{eq:phi} for each node with $r=1$. The $2^n$ model is on the left and the $n$-intertwined model is on the right. Blue indicates the agent is healthy and red indicates the agent is infected. This figure gives a snapshot of the system at time 40. For a video of this simulation please see \href{https://youtu.be/-LmPj7oynLs}{youtu.be/-LmPj7oynLs}.}
\label{fig:compDrift}
\end{figure}

For another comparison we will use a piecewise constant drift so that the agents remain confined to a fixed region. Without loss of generality, let the constrained region be a hypercube $l^d$, where $d$ is the dimension of the space, centered at some point $z_c$. That is, the dynamics follow \eqref{eq:phi} but instead of a constant $\phi$ term  for each agent we have,
\begin{equation}\label{eq:phipiece}
\phi_k =\begin{cases}
    -\phi_k, & \text{if } z_k=z_{c_k}+l/2 \text{ or } z_k=z_{c_k}-l/2 \\
   \ \ \phi_k ,             & \text{otherwise},
\end{cases}
\end{equation}
for each dimension $k = 1,\dots, d$. That is, if an agent hits a boundary, the velocity of the agent in the dimension corresponding to that boundary flips sign. 
As illustrated in Figure \ref{fig:compAt1}, the upper bounding nature of the $n$-intertwined Markov chain model leads to an inaccurate approximation, as the $2^n$ model reaches the DFE but the $n$-intertwined model does not   and does not appear to be tending towards the DFE. 

\begin{figure}
    \centering
      \includegraphics[width=.95\columnwidth]{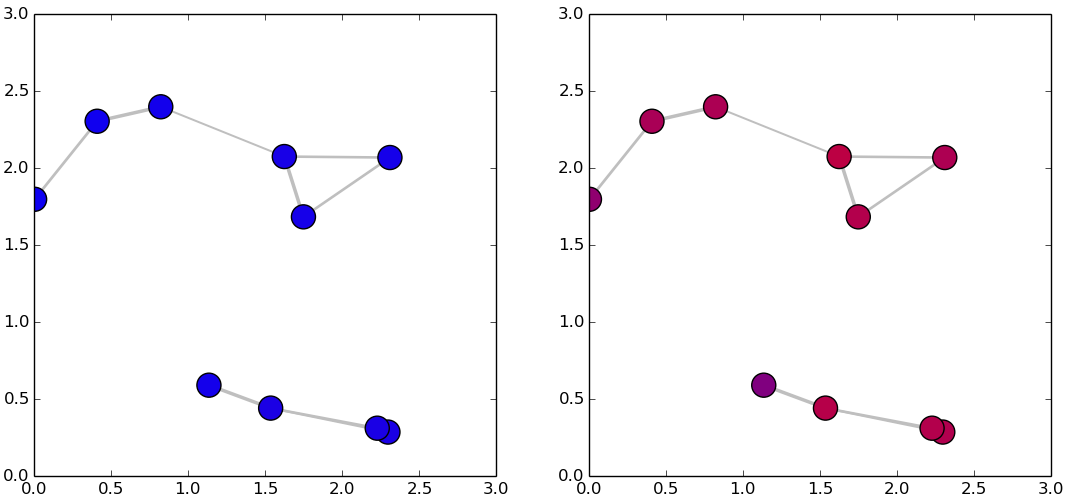}
      \caption{This system has piecewise constant drift as explained in \eqref{eq:phipiece} for each node with $r=1.5$. The $2^n$ model is on the left and the $n$-intertwined model is on the right. Blue indicates the agent is healthy and red indicates the agent is infected. This figure gives a snapshot of the system at time 40.  For a video of this simulation please see \href{https://youtu.be/BMn4FGnBZX0}{youtu.be/BMn4FGnBZX0}.}
\label{fig:compAt1}
\end{figure}


\subsection{Exploratory Time--Varying Simulations and Extensions}


Considering different dynamics models provides us with insights leading to several corollaries of Theorem \ref{thm:atorigin}.

\subsubsection{Constant Drift}

If the agents have constant (non-equal) drift, defined in \eqref{eq:phi}, they will eventually float away from each other far enough that, assuming they have non-zero healing rate, the disease will be eradicated. 
This is illustrated in the simulations depicted in Figure \ref{fig:compDrift}. 

This behavior is captured in the following corollary:
\begin{corollary}\label{cor:atT}
If $B = \beta I$ and $A(t)$ is symmetric,  piecewise continuous in $t$, and bounded $\forall t \geq T$, and for some fixed $T$, $\sup_{t\geq T} \lambda_1(BA(t)-D)<0$, then $\forall t \geq T$ the DFE is globally exponentially  stable.
\end{corollary}
\begin{proof}
Let $\hat{t} = t-T$. The result follows immediately from applying Theorem \ref{thm:atorigin} to $\dot{p}(\hat{t})$ for $\hat{t}\geq 0$. 
\end{proof}


\subsubsection{Piecewise Constant Drift}

Consider the $n$-intertwined model with piecewise constant drift illustrated by the right hand side of Figure \ref{fig:compAt1}. 
At several time instances the disease appears to be approaching the DFE, but when the graph structure changes, due to certain agents crossing paths, the system is pushed away from the DFE. 
The behavior of the system is consistent with  Theorem \ref{thm:DFEasymp} and Lemma \ref{lem:unst}; the system is constantly fluctuating between approaching the DFE and tending away from it. 

\begin{remark}\label{rem:obv}
Suppose $B=\beta I$, $D = \delta I$,  and $A(t)$ is determined by \eqref{eq:at}, \eqref{eq:phi}, and \eqref{eq:phipiece} with a finite $l$. If $n^2-n < \frac{\delta}{\beta}$ then the DFE is globally asymptotically stable. This follows from the Frobenius norm bounding the spectral radius, the bound on $a_{ij}(t)$ imposed by \eqref{eq:at}, and Theorem \ref{thm:atorigin}.
\end{remark}

\noindent This remark states that if $ \frac{\delta}{\beta}$ is large enough then the virus will be eradicated, appealing to Theorem \ref{thm:atorigin}. 
\begin{remark}
Simulations show  for systems with dynamics determined by \eqref{eq:at}, \eqref{eq:phi}, and \eqref{eq:phipiece}, if $\lambda_1(BA(t)-D)<0$ for more than half the time than the system  still converges to the DFE; so the bound in Remark \ref{rem:obv} is very conservative. 
\end{remark}


\subsection{Behavior of Systems}
As we have seen, for the time--varying graphs  considered herein, the sign of $s_1(BA(t)-D)$  can easily change 
(see Figure \ref{fig:quareig}). 
It is well--known that the origin is an unstable equilibrium if $s_1(BA(t)-D)>0$, by Lyapunov's indirect method:
\begin{lemma} \label{lem:unst}\emph{\cite{khanafer2014stability}}
The origin is unstable when $s_1(BA-D)>0$.
\end{lemma}
Consequently, if the disease is close to disappearing and then $s_1(BA(t)-D)$ becomes positive, the disease can easily reemerge as long as $p(t)\neq 0$. That is, these systems may jump back and forth between tending towards or away from the DFE. There are several ways to eliminate this behavior: inhibit the agents' movement dynamics, make their movement identical (practically reducing it to the static case), spread the agents out
, or eradicate the disease  completely. 
\begin{remark}\label{cor:free}
If there exists a $T$ such that $p(T)=0$, then $p(t) = 0$ $ \forall t\geq T$ and therefore the graph structure does not affect the DFE.
By Lemma \ref{lem:unst} if $s_1(BA(t)-D)>0$ the origin is unstable. However if there is no disease at time $T$, that is $p(T)=0$, then it is irrelevant if the origin is unstable because no infection exists in the system. Therefore, no bound is necessary on $s_1(BA(t)-D)$ for $t>T$.
\end{remark}
\noindent That is, if there is a long enough time period $ [t_1, T]$, where  $s_1(BA(t)-D)<0$, such that $p(T)=0$, then it does not matter if $s_1(BA(t)-D)\geq 0$ for $t\geq T$. This motivates the following discussions. 

\subsubsection{Quarantine}
\begin{figure}
    \centering
    \begin{subfigure}[b]{0.49\textwidth}
      \includegraphics[width=\textwidth]{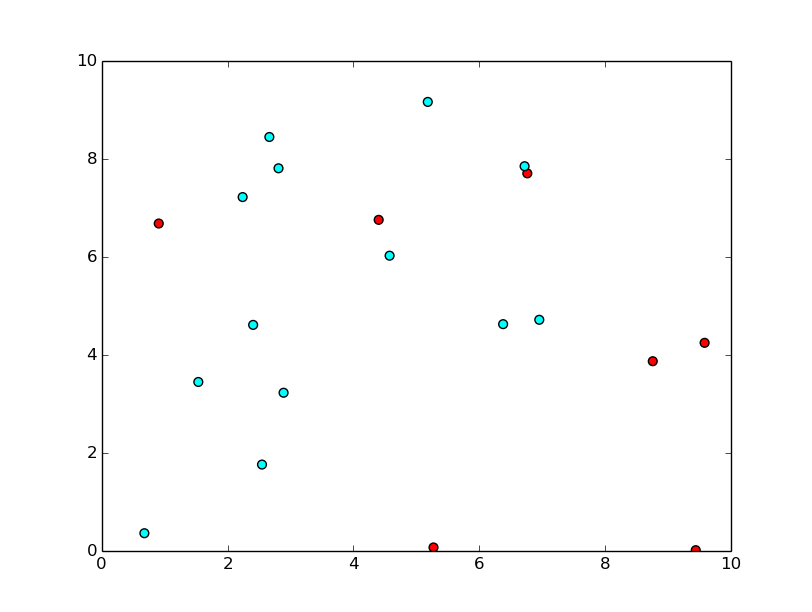}
      \caption{The system at time zero.}
      \label{fig:quar0}
    \end{subfigure}
    \hfill
    \begin{subfigure}[b]{0.49\textwidth}
      \includegraphics[width=\textwidth]{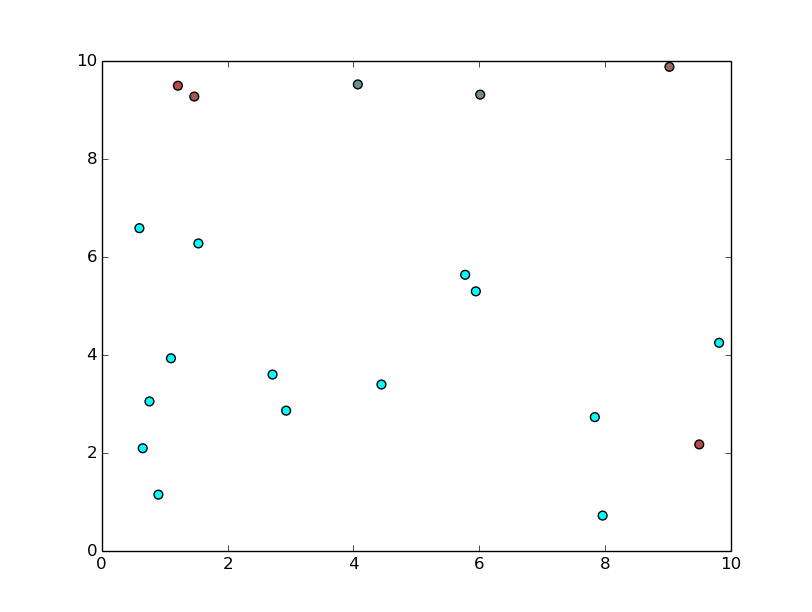}
      \caption{The system at time 400.}
      \label{fig:quar400}
    \end{subfigure}
    \caption{This system has piecewise constant drift as explained in \eqref{eq:phipiece} for each node and evolves for 400 time steps. After 50 time steps a quarantine is implemented, limiting the agents to certain regions. This separates the sick and the more prone to get sick agents from the others. 
For a video of this simulation please see \href{http://youtu.be/NfskXS83FHI}{youtu.be/NfskXS83FHI}.}
\label{fig:quar}
\end{figure}
If the system is too dense then the agents can be split into different groups, spatially bounding them to separate regions. This technique is called a quarantine \cite{wan2007network,enns2012optimal}, and  is reflected in the model by removing edges in the graph, i.e. setting specific elements of the $A$ matrix to zero. 
The quarantine essentially imposes a block diagonal structure on the $A$ matrix, given that the states are properly ordered; this restricts  interaction between certain agents, which can clearly reduce the spread of a virus.

A quarantine is difficult to implement, as was witnessed recently with the Ebola virus \cite{cnnEbola}. A quarantine could also be effected by less costly implementations, such as, decreasing human contact via limiting handshakes and other greetings, instilling good habits of covering mouths, etc. Without restricting movement, these measures would decrease the weight of the links between agents, which would be reflected in the model by decreasing the values of $a_{ij}$. 

For the following discussion we assume $A$ is symmetric. If the $a_{ij}$'s can be restricted such that, for $\epsilon_i<0$, 
\begin{align*}
 \sup_{t \geq 0} \sum_{i = 1}^n a_{ij}(t) \leq \epsilon_i + \frac{\delta_i}{\beta} \ \forall i, \\ 
 \Rightarrow \beta \sup_{t \geq 0}\: \sum_{i = 1}^n a_{ij}(t) -\delta_i \leq \epsilon_i  \ \forall i.
\end{align*}
Therefore, by the Gershgorin Disc Theorem \cite{horn2012matrix},
\begin{equation*}
\sup_{t\geq0} \lambda_1(BA(t)-D)<0.
\end{equation*}
So by Theorem 1 
the disease will be eradicated in exponential time.

We can implement a quarantine on the piecewise constant drift case by imposing a block diagonal structure, limiting the movement of certain agents so that they do not interact with others. Consider a system with 20 agents, originally confined to a $40\times 40$ box with certain random initial conditions (see Figure \ref{fig:quar0}). After 50 time steps a quarantine is imposed, limiting some agents to the region $[ 0, 25]\times [0,25]$ and exiling the rest of the agents to the outside boundary. This separates the sick and the more prone to get sick agents from the others. One set of agents is tending towards the DFE and the other is not, which is consistent with the maximum eigenvalue plot in Figure \ref{fig:quareig}.
 This leads to the following corollary:
\begin{corollary}\label{cor:quar}
Imposing a block diagonal graph structure, such that $A(t) = diag(A_1(t), \dots, A_q(t))$, makes the DFE globally exponentially stable if 
$\lambda_1(B_l A_l(t)-D_l)<0$ for all $t\geq 0$ and $l = 1,\dots , q$, with $B_l = \beta_l I$ and $A(t)$ is symmetric,  piecewise continuous in $t$, and bounded $\forall t \geq 0$.
\end{corollary}
\begin{proof}
Since $A(t)$ is block diagonal and $B$ and $D$ are diagonal, $p(t)$ can be partitioned into $q$ pieces. Then the result follows immediately from applying Theorem \ref{thm:atorigin} to each $\dot{p}_l(t)$.
\end{proof}

\begin{figure}
\centering
    \includegraphics[scale=.58]{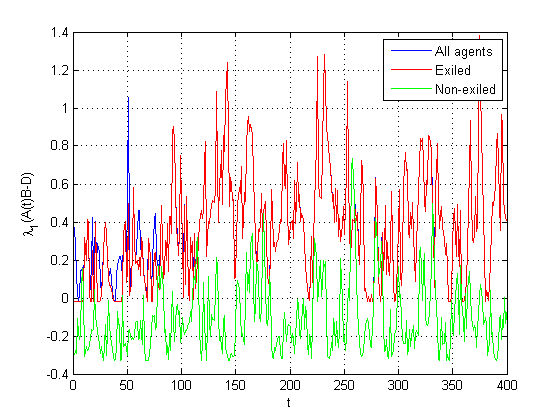}
    \caption{This is the plot of the maximum eigenvalues  the system shown in Figure \ref{fig:quar}. The blue line is the total system, the red line is the outside (sick) group, and the green line is the inside (healthy) group. Notice the max eigenvalue of the inside group is  below zero on average.}
\label{fig:quareig}
\end{figure}

This approach does not consider any reintegration process for healed agents. For this to be a feasible, effective control technique, healed agents would have to be allowed to rejoin the general population. Therefore, appealing to Corollary \ref{cor:free}, once time $T$ was reached such that $p(T)= 0$, restrictions would have to be lifted. More interestingly, an agent by agent policy could be implemented so that once an agent was healed, i.e. $p_i(t) = 0$ for some $t$, re-admittance would be permitted, contingent on them remaining healthy during re-entrance.

\section{Conclusion}\label{sec:con}

We have extended well--studied SIS models to the time--varying graph structure case. Prior to this, the dynamic modeling of such systems was mainly focused on networks with static graph structures. This extension makes the models more realistic and gives us better insight into disease propagation in most settings. We provided a stability analysis of the time--varying, weighted, undirected and directed, deterministic and stochastic $n$-intertwined Markov chain models for the DFE. We compared the $2^n$ model and the $n$-intertwined Markov chain model for both static and dynamic graph structures via simulation, showing the weaknesses and strengths of the mean field approximation. We provided various different insightful network simulations with links to videos, which inspired a number of corollaries and remarks. We also explored quarantine control via simulation. 

For future work we would like to further explore  the development of  optimal control problems. 
Exploring the existence and stability properties of the NDFE trajectory for the generic time--varying model  is a problem that requires more investigation. 
We would also like to extend all these ideas to large--scale simulations with at least thousands of agents.

\section{Acknowledgements}

The authors wish to thank Geir Dullerud, Daniel Liberzon, Liming Feng, and Ji Liu for several insightful discussions related to this work.


\bibliographystyle{IEEEtran}
\bibliography{bib}

\end{document}